\documentclass[11pt]{article}
\usepackage{amsmath}
\usepackage{latexsym}
\usepackage{graphicx}
\usepackage{amssymb}
\begin{document}
\makeatletter
\def\@begintheorem#1#2{\trivlist \item[\hskip \labelsep{\bf #2\ #1.}] \it}
\def\@opargbegintheorem#1#2#3{\trivlist \item[\hskip \labelsep{\bf #2\ #1}\ {\rm (#3).}]\it}
\makeatother
\newtheorem{thm}{Theorem}[section]
\newtheorem{alg}[thm]{Algorithm}
\newtheorem{conj}[thm]{Conjecture}
\newtheorem{lemma}[thm]{Lemma}
\newtheorem{defn}[thm]{Definition}
\newtheorem{cor}[thm]{Corollary}
\newtheorem{exam}[thm]{Example}
\newtheorem{prop}[thm]{Proposition}
\newenvironment{proof}{{\sc Proof}.}{\rule{3mm}{3mm}}

\title{Honeycomb Toroidal Graphs}
\author{Brian Alspach\\School of Mathematical and Physical Sciences\\
University of Newcastle\\Callaghan, NSW 2308\\
Australia\\brian.alspach@newcastle.edu.au}
\date{}

\maketitle

\begin{abstract} Honeycomb toroidal graphs are trivalent Cayley graphs on generalized dihedral
groups.  We examine the two historical threads leading to these graphs, some of the properties
that have been established, and some open problems.
\end{abstract}

\medskip


\smallskip


\section{Introduction}

This paper discusses a family of graphs called {\it honeycomb toroidal graphs}.  They have arisen in
two distinct settings which are discussed in the next two sections.  This is followed by an examination
of some of their properties and some parameters of interest.  Several open problems also are
mentioned.

Given the disparate subject areas employing graphs as models, there are a variety of concepts
for which different terms are used across disciplines.  Thus, we shall mention some terminolgy used
in this paper.  A {\it graph} has neither loops nor multiple edges.  The {\it valency} of a vertex $v$,
denoted $\mathrm{val}(v)$, is the number of edges incident with $u$.  The {\it order} of a graph
is the cardinality of its vertex set and the {\it size} of a graph is the cardinality of its edge set.

\begin{defn}\label{cay}{\em Let $G$ be a finite group and $S\subset G$ satisfying $1\not\in S$
and $s\in S$ if and only if $s^{-1}\in S$.  The} Cayley graph on $G$ with connection set $S$,
{\em denoted $\mathrm{Cay}(G;S)$, has the elements of $G$ as its vertex set and has an edge
joining $g$ and $h$ if and only if $h=gs$ for some $s\in S$.}
\end{defn}

A {\it path of length} $\ell$ in a graph is a subgraph consisting of a sequence $v_0,v_1,\ldots,v_{\ell}$
of $\ell+1$ distinct vertices such that the edge $v_iv_{i+1}$ belongs to the path for
$i=0,1,\ldots,\ell-1$.  A {\it cycle of length}
$\ell$ is a connected subgraph of size $\ell$ in which every vertex has valency 2.  Cycles are denoted
by a sequence of vertices as they occur along the cycle with the convention that the first vertex
and the last vertex of the sequence are the same in order to distinguish it from a path.
A {\it Hamilton cycle} in a graph is a cycle containing every vertex of the graph, and a {\it Hamilton
path} is a path containing every vertex.

\section{Algebraic And Topological Viewpoint}

Altshuler \cite{A1} considered three families of equivelar maps on the torus and was able to show that
every graph in two of the families possesses a Hamilton cycle, but was unable to do so for the other
family.  The latter family consists of the equivelar maps with Schl\"afli type $(6,3)$, that is, the boundaries
of all the faces are 6-cycles and vertices all have valency 3.  Many of these graphs, but not all, are Cayley
graphs on the appropriate dihedral group.  So this problem arising in topological graph theory impinges
on another problem which has drawn considerable attention for fifty years, namely, does every
connected Cayley graph of order at least three have a Hamilton cycle?

The answer to the preceding question for Cayley graphs on abelian groups was known to be yes
as early as the first edition of Lov\'asz's book entitled {\sl Combinatorial Problems and Exercises}
\cite{L1}.  However, a much stronger result by Chen and Quimpo \cite{C2} appeared in 1981.
Their theorem follows two definitions.  A graph $X$ is {\it Hamilton-connected} if for every pair
of vertices $u$ and $v$ in $X$ there is a Hamilton path whose terminal vertices are $u$ and $v$.
A bipartite graph $X$ is {\it Hamilton-laceable} if the same property holds for any two vertices in
opposite parts.

\begin{thm} If $X$ is a connected Cayley graph of valency at least 3 on an abelian group, then
$X$ is Hamilton-connected unless it is bipartite in which case it is Hamilton-laceable.
\end{thm}

Note that the preceding theorem implies that every edge of a connected Cayley graph on an
abelian group belongs to a Hamilton cycle.  If the valency is at least 3, it is implied by the
theorem.  If the valency is 2, the graph is a Hamilton cycle. 

The dihedral group is close to being abelian in the sense that the dihedral group $D_n$ of order
$2n$ contains an abelian subgroup of index 2, that is, has order $n$.  In fact, it still is not known
whether every connected Cayley graph on $D_n$ is hamiltonian in spite of the efforts of a
non-trivial number of people working on the problem for the last forty plus years.

As we shall see soon, when considering Cayley graphs on dihedral groups, those for which the
connection set consists of three reflections turn out to be crucial.  Let's now take a closer look at
these particular graphs.

Throughout this paper we let $D_n$ denote the dihedral group of degree $n$ and order $2n$.
We visualize the group as the symmetries of a regular $n$-gon.  So the group is generated by
an element $\rho$ of order $n$ (it rotates the $n$-gon cyclically) and a reflection $\tau$.
Thus, $|\tau|=2$ and $\tau\rho\tau=\rho^{-1}$.  The cyclic subgroup $\langle\rho\rangle$
has index 2 in $D_n$.  Note that the coset $\langle\rho\rangle\tau$ consists of $n$ reflections.
When $n$ is odd, the $n$ reflections are the only involutions in $D_n$, whereas, $\rho^{n/2}$
also is an involution when $n$ is even.

We are interested in Cayley graphs on $D_n$ whose connection sets consist of three reflections.
Let $S=\{\rho^i\tau,\rho^j\tau,\rho^k\tau\}$, where $0\leq i<j<k<n$.  It is clear that the
connection set $\{\tau,\rho^{j-i}\tau,\rho^{k-i}\tau\}$ produces an isomorphic Cayley graph.
Hence, we shall assume the connection set has the form $S=\{\tau,\rho^i\tau,\rho^j\tau\}$, where
$0<i<j<n$.

The graph $X=\mathrm{Cay}(D_n;S))$ is connected if and only if it is the case that $\mathrm{gcd}(n,i,j)=1$.  So
dealing with connectivity is straightforward.  The subgraph generated using just $\tau$ and
$\rho^i\tau$ consists of $m$ cycles of length $2n/m$, where $m=\mathrm{gcd}(n,i)$.  The case
in which we are most interested is when $m>1$ and $X$ is connected.  This means that the
element $\rho^j\tau$ generates edges that connect the $m$ cycles to form a single component
for $X$.  We want to take a careful look at these graphs to see how to represent them nicely.

The vertices of $\langle\rho\rangle$ are cyclically labelled $1,\rho,\rho^2,\ldots,\rho^{n-1}$
and those of $\langle\rho\rangle\tau$ are cyclically labelled $\tau,\rho\tau,\rho^2\tau,\ldots,
\rho^{n-1}\tau$.  The $m$ cycles have the properties that they have even length at least 4, and
the vertices alternate between belonging to $\langle\rho\rangle$ and $\langle\rho\rangle\tau$.
Let $\rho^j\tau$ generate an edge joining a vertex of $\langle\rho\rangle$ in a cycle $C_1$ to a
vertex of $\langle\rho\rangle\tau$ in another cycle $C_2$. The distance (under the cyclic labellings)
to the next element of $\langle\rho\rangle$ along $C_1$ is the same as the distance to the next
element of $\langle\rho\rangle\tau$ along $C_2$.  Hence, these two vertices also are joined by
an edge generated by $\rho^j\tau$.  This is true for all the edges generated by $\rho^j\tau$
joining the cycles together.

Therefore, we may label the vertices of the graph as $u_{i,j}$, $0\leq i\leq m-1$ and
$0\leq j\leq n-1$ so that the following are the edges:
\begin{itemize}
\item $u_{i,j}u_{i,j+1}$ for $i=0,1,\ldots,m-1$ and $j=0,1,\ldots,n-1$, where the second
subscript is reduced modulo $n$ (these are called {\it vertical edges});
\item $u_{i,j}u_{i+1,j}$ for $i=0,1,\ldots,m-2$ and all $j$ such that $i+j$ is odd (these are
called {\it flat edges}; and
\item $u_{m-1,j}u_{0,j+\ell}$, where $m,j,\ell$ all have the same parity (these are called
{\it jump edges}.
\end{itemize}

As we label the columns from left to right, it is clear that we may assure that the edges
between successive columns are flat.  However, once the last column is labelled the only
feature we know about the edges from the last column back to the first column is that they
have the same change in the second coordinate, that is, they have the same jump.  We
now have a straightforward description of these graphs.  They are called honeycomb
toroidal graphs and are denoted $\mathrm{HTG}(m,n,\ell)$, where $m$ is the number of column
cycles, $n$ is the length of the column cycles so that $n\geq 4$ and is even, and $\ell$ is
the jump from the last column back to the first.  Figure 1 shows the 

\begin{picture}(200,260)(-120,-40)
\multiput(0,0)(20,0){2}{\multiput(0,0)(0,20){10}{\circle*{5}}}
\multiput(0,0)(20,0){2}{\line(0,1){180}}
\multiput(0,20)(0,40){5}{\line(1,0){20}}
\multiput(40,0)(0,20){10}{\circle*{5}}
\put(40,0){\line(0,1){180}}
\multiput(20,0)(0,40){5}{\line(1,0){20}}
\multiput(60,0)(0,20){10}{\circle*{5}}
\put(60,0){\line(0,1){180}}
\multiput(40,20)(0,40){5}{\line(1,0){20}}
\multiput(-10,-10)(80,0){2}{\line(0,1){200}}
\multiput(-10,-10)(0,200){2}{\line(1,0){80}}
\multiput(0,-10)(20,0){4}{\line(0,1){10}}
\multiput(0,180)(20,0){4}{\line(0,1){10}}
\multiput(-10,0)(0,40){5}{\line(1,0){10}}
\multiput(60,0)(0,40){4}{\line(1,4){10}}
\put(60,160){\line(1,4){7.5}}
\put(67.5,-10){\line(1,4){2.5}}
\multiput(43,-10)(0,200){2}{\vector(1,0){8}}
\multiput(-10,63)(80,0){2}{\vector(0,1){8}}
\put(-77,-35){{\sc Figure 1}: $\mathrm{HTG}(4,10,2)$ embedded on the torus}
\end{picture}

\medskip

\noindent honeycomb toroidal graph $\mathrm{HTG}(4,10,2)$ embedded on the torus.

Figure 1 demonstrates clearly how $\mathrm{HTG}(m,n,\ell)$ may be embedded on a torus for
any choice of the parameters.  Even though these graphs all have nice embeddings on the
torus, they are slightly misnamed in that they are not all toroidal graphs.  This turn of events
comes about because in order for a graph that embeds on the torus to be toroidal it must be
non-planar.  The graphs $\mathrm{HTG}(2,n,0)$, for all even $n\geq 4$, and $\mathrm{HTG}
(2,4,2)$ are, in fact, planar graphs.  All others are non-planar.  Nevertheless, we shall refer
to all graphs in the family as honeycomb toroidal graphs  

\section{Network Topology Viewpoint}

Network topology refers to methods used to connect objects together to perform certain tasks.
For example, connecting computers together to form a computer network or connecting processors
within a single computer fall within the area.  Some desirable properties are small valency so that the number of
direct connections is not too big and symmetry meaning that all the vertices are essentially the
same which allows local algorithms to be the same at each vertex.

One approach is to start with tesselations of the plane by regular polygons.  These have an
infinite number of vertices so that some modifications are required.  One such modification is
to bound a finite region of a tesselation with a ``nice polygon'' to obtain a finite graph.  The
latter graph is called a {\it mesh}.  The mesh is not regular but the addition of a few edges may
result in a graph that is not only regular but also is vertex-transitive.

The tesselation of the plane by regular hexagons is one source for which this was done.
Stojmenovic \cite{S1} suggested three bounding types of polygons to obtain meshes: a
hexagonal polygon, a square polygon and a rhombic polygon.  He then determined ways to
add edges so that all vertices have valency 3 and the resulting graph is vertex-transitive.

\bigskip

\includegraphics[scale=.5]{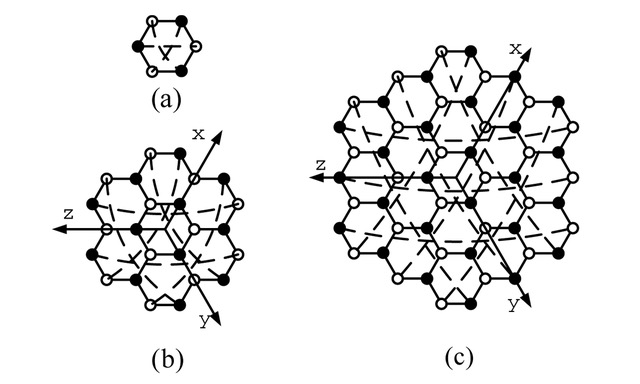}

\begin{center}
{\sc Figure 2}
\end{center}

\bigskip

Figure 2 shows the three smallest graphs obtained by Stojmenovic using a hexagonal bounding
polygon.  We need to examine the viewpoint in some detail because this formed the foundation
for the way subsequent researchers in the area developed the ideas.  He called the graph in
Figure 2(a) a {\it hexagonal torus of size 1}.  The graph in Figure 2(b) he called the hexagonal torus
of size 2.  Thus, to increase the size by 1 we add a ring of hexagons around the current graph.
This is a very geometric way of building the graphs.

His use of a square as a bounding polygon has since been extended to using a rectangle, and
his use of a rhombus for a bounding polynomial has been extended to using a parallelogram.
All three types of graphs share the property that they arise geometrically.  They are, in fact,
very special honeycomb toroidal graphs as domonstrated in the next result.

\begin{prop} The hexagonal torus of size $m$ is $\mathrm{HTG}(m,6m,3m)$ for $m\geq 1$.
The rectangular torus is $\mathrm{HTG}(m,n,0)$ for even $m\geq 2$.  The parallelogramic
torus is $\mathrm{HTG}(m,n,m')$, where $m'\equiv m(\mbox{mod }n)$ and $0\leq m'<m$.
\end{prop}

Some comments about terminology are in order.  Because of the way the network topology
approach developed these graphs, they were understandably viewed as special.  Thus, when
it was discovered \cite{C1} how to broaden the construction, the term {\it generalized honeycomb torus}
was adapted and appears in many papers.  However, we object to this terminology for two
reasons leading to the term honeycomb toroidal graph being used.  

The first objection is because the torus is a closed orientable surface of genus one and
even though these graphs have nice embeddings on the torus, the graphs themselves
should not be called tori.  The second objection arises because we have seen that the only
differences between them come from changing three descriptive parameters.  There is no
particular set of parameters that is special and the term `generalized' is inappropriate.  

\section{Hamiltonicity}

Hamiltonicity refers to various properties of graphs revolving around Hamilton paths and Hamilton cycles.
We consider two properties in this section.  The first is the hamiltonian property, that is, does $\mathrm{HTG}
(m,n,\ell)$ have a Hamilton cycle?  The second property is Hamilton laceability, that is, is every
honeycomb toroidal graph Hamilton-laceable? 

The answer to the first question is yes and was proven in \cite{Y2}.  We give a short proof
of this result but before doing so we discuss a useful constructive technique for honeycomb toroidal
graphs.

Consider a graph with vertex set $\{u_{i,j}:0\leq i\leq 2\mbox{ and }0\leq j\leq n\}$ and edges
$u_{0,t_1}u_{1,t_1};u_{0,t_2}u_{1,t_2};\ldots;u_{0,t_k}u_{1,t_k}$, where $0\leq t_1<t_2<
\cdots<t_k\leq n$.  So the graph consists of an $(n+1)\times 3$ array of vertices and some flat
edges between column 0 and column 1.  Extend each edge $u_{0,t_a}u_{1,t_a}$ to a path
from $u_{0,t_a}$ to $u_{2,t_a}$ by adding the vertical path from $u_{1,t_a}$ down to
$u_{1,1+t_{a-1}}$ followed by the edge $u_{1,1+t_{a-1}}u_{2,1+t_{a-1}}$ and then back
up column 2 to $u_{2,t_a}$.  We then obtain paths from $u_{0,t_a}$ to $u_{2,t_a}$ that use
all the vertices of columns 1 and 2.  This operation is called the {\it vertical downward fill} for
columns 1 and 2.  The {\it vertical upward fill} is defined in an obvious analogous manner.  
These operations are most clearly seen by looking at Figure 3 which shows an example of
both vertical fills and makes everything obvious. 

\begin{picture}(260,290)(0,-40)
\multiput(0,0)(30,0){3}{\multiput(0,0)(0,20){12}{\circle*{5}}}
\multiput(120,0)(30,0){3}{\multiput(0,0)(0,20){12}{\circle*{5}}}
\multiput(240,0)(30,0){3}{\multiput(0,0)(0,20){12}{\circle*{5}}}
\multiput(0,60)(120,0){3}{\line(1,0){30}}
\multiput(0,120)(120,0){3}{\line(1,0){30}}
\multiput(0,140)(120,0){3}{\line(1,0){30}}
\multiput(0,200)(120,0){3}{\line(1,0){30}}
\multiput(150,0)(30,0){2}{\line(0,1){60}}
\multiput(150,80)(30,0){2}{\line(0,1){40}}
\put(150,140){\line(1,0){30}}
\multiput(150,160)(30,0){2}{\line(0,1){40}}
\put(150,80){\line(1,0){30}}
\put(150,160){\line(1,0){30}}
\put(150,220){\line(1,0){30}}
\qbezier(150,0)(130,110)(150,220)
\qbezier(180,0)(200,110)(180,220)
\multiput(270,60)(30,0){2}{\line(0,1){40}}
\multiput(270,140)(30,0){2}{\line(0,1){40}}
\multiput(270,200)(30,0){2}{\line(0,1){20}}
\multiput(270,0)(30,0){2}{\line(0,1){40}}
\put(270,40){\line(1,0){30}}
\put(270,100){\line(1,0){30}}
\put(270,120){\line(1,0){30}}
\put(270,180){\line(1,0){30}}
\qbezier(270,0)(250,110)(270,220)
\qbezier(300,0)(320,110)(300,220)
\put(130,-30){\sc Figure 3}
\end{picture}

\begin{thm}\label{ham} Every honeycomb toroidal graph is hamiltonian.
\end{thm}
\begin{proof} Claim: If $\mathrm{HTG}(m,n,\ell)$, $m\geq 2$, is hamiltonian, then the graph
$\mathrm{HTG}(m+2,n,\ell)$ also is hamiltonian. It is easy to see that there must be at least one
flat edge between column 0 and column 1 in any Hamilton cycle.  So let $u_{0,j_1}u_{1,j_1},u_{0,j_2}u_{1,j_2},
\ldots,u_{0,j_t}u_{1,j_t}$, $0<j_1<j_2<\cdots<j_t<n$, be the flat edges between column 0
and column 1 in some Hamilton cycle.

Subdivide each flat edge with two new vertices in each edge.  Remove the central edge in each of the
subdivided edges and use vertical fills between the two new columns to obtain a Hamilton cycle
in $\mathrm{HTG}(m+2,n,\ell)$.

Thus, it suffices to prove that $\mathrm{HTG}(2,n,\ell)$ and $\mathrm{HTG}(3,n,\ell)$ are hamiltonian.
Consider $m=2$ first.  For each even $i$, let $P_i$ be the 4-path \[u_{0,i}u_{0,i+1}u_{1,i+1}u_{1,i}
u_{0,i+\ell}.\]  Start a path with $P_0$ followed by $P_{\ell}$ followed by $P_{2\ell}$ and so on.
This eventually closes off to form a cycle.  If the cycle is a Hamilton cycle, we are done.  If it is not
a Hamilton cycle, then perform vertical fills upwards on each flat edge (removing the flat edge)
to obtain a Hamilton cycle.

Now consider $\mathrm{HTG}(1,n,\ell)$.  The column itself is a Hamilton cycle that uses none of the
jump edges.  By Smith's Theorem \cite{T1} there is a second Hamilton cycle $C$ and it must use some jump
edges.  Each jump edge has the form $u_{0,i}u_{0,j}$ with $i$ odd and $j=i+\ell$ even.  Let $0< i_1<i_2<
\cdots <i_t<n$ be the odd subscripted vertices of the jump edges in $C$.  Add two columns whose vertices
are labelled conventionally.  Replace the jump edge $u_{0,i_r}u_{0,i_r+\ell}$ with the jump edge
$u_{2,i_r}u_{0,i_r+\ell}$ and add the flat edge $u_{0,i_r}u_{1,i_r}$ for each $i_1,i_2,\ldots,i_t$.
Now use vertical fills between columns 1 and 2 to obtain a Hamilton cycle in $\mathrm{HTG}(3,n,\ell)$
completing the proof.    \end{proof}

\medskip

The second question is not yet settled and we state it as a research problem.

\smallskip

{\bf Research Problem 1}.  Is every $\mathrm{HTG}(m,n,\ell)$ Hamilton-laceable?

\smallskip

Some comments about the preceding problem are in order.  It is a signifcant problem because an affirmative
answer implies that the family of connected Cayley graphs of valency at least 3 on generalized dihedral
groups satisfies the conclusions of the Chen - Quimpo Theorem.  A special conclusion from this,
of course, is that every connected Cayley graph on a dihedral group is hamiltonian.  The fact that the
latter conclusion still is unsettled is a frustrating situation.

There has been some progress on Research Problem 1.  In \cite{A2} it is proved that $\mathrm{HTG}(m,n,\ell)$
is Hamilton-laceable whenever $m$ is even.  This leaves the case that $m$ is odd.  A few special cases
for $m=1$ are solved in \cite{A5}.  The following result
is due to McGuinness \cite{M1}. His manuscipt contains a long proof and was not published.  Consequently,
we provide a short proof here for convenience.

\begin{thm} If $\mathrm{HTG}(1,n,\ell)$ is Hamilton-laceable, then $\mathrm{HTG}(m,n,\ell)$ is Hamilton-laceable
for all odd $m\geq 1$.
\end{thm}
\begin{proof} Using the same method as in the proof of Theorem \ref{ham}, it is easy to show that
if $\mathrm{HTG}(3,n,\ell)$ is Hamilton-laceable, then $\mathrm{HTG}(m,n,\ell)$ is Hamilton-laceable
for all odd $m\geq 3$.  This reduces the proof to showing that $\mathrm{HTG}(1,n,\ell)$ being
Hamilton-laceable implies that $\mathrm{HTG}(3,n,\ell)$ is Hamilton-laceable.

Assume that $\mathrm{HTG}(1,n,\ell)$ is Hamilton-laceable.  Let $P'$ be a Hamilton path in
$\mathrm{HTG}(1,n,\ell)$ from $u_{0,0}$ to $u_{0,j}$ using at least one jump edge.
Because $\mathrm{HTG}(1,n,\ell)$ is bipartite, $j$ must be odd and the subscripts of the
end vertices of jump edges have opposite parity.

Project $P'$ into the edge set of $\mathrm{HTG}(3,n,\ell)$ as follows.  If $u_{0,j}u_{0,j+1}$
is an edge of $P'$, where the subscripts are treated modulo $n$, then $u_{0,j}u_{0,j+1}$
is an edge of the projection in $\mathrm{HTG}(3,n,\ell)$.  If $u_{0,j}u_{0,k}$ is a jump
edge in $P'$ with $j$ odd and $k$ even, then $u_{2,j}u_{0,k}$ is an edge in the projection
in $\mathrm{HTG}(3,n,\ell)$. 

Let $u_{2,j_1},u_{2,j_2},\ldots,u_{2,j_t}$ be the end vertices in column 2 of the projected jump edges,
where $0<j_1<j_2<\cdots<j_t<n$.  Now add the flat edges $u_{0,j_a}u_{1,j_a}$ for $a=1,2,\dots,t$. 
Vertical fills between columns 1 and 2 yield a Hamilton path from $u_{0,0}$ to $u_{0,i}$.
Furthermore, if we also add the flat edge $u_{0,j}u_{1,j}$ and then do the vertical fills between
columns 1 and 2, we obtain a Hamilton path from $u_{0,0}$ to $u_{2,j}$. 

From the preceding, we see that whenever there is a Hamilton path from $u_{0,0}$ to $u_{0,j}$
in $\mathrm{HTG}(1,n,\ell)$ using at least one jump edge, then there are Hamilton paths from
$u_{0,0}$ to both $u_{0,j}$ and $u_{2,j}$ in $\mathrm{HTG}(3,n,\ell)$.  So the presence
of jump edges is crucial.

A Hamilton path in $\mathrm{HTG}(1,n,\ell)$ from $u_{0,0}$ to $u_{0,j}$ must use jump
edges if $j$ is neither 1 nor $n-1$. Because $u_{0,0}u_{0,1}\cdots u_{0,n-1}u_{0,0}$ is a
Hamilton cycle in $\mathrm{HTG}(1,n,\ell)$, there is another Hamilton cycle $C$, by Smith's
Theorem \cite{T1}, using the edge $u_{0,0}u_{0,1}$.  Clearly $C$ must have at least one jump
edge.  The same argument applies to the edge $u_{0,0}u_{0,n-1}$ 
Therefore, for each $u_{0,j}$, $j$ odd, there is a Hamilton path in $\mathrm{HTG}(3,n,\ell)$
from $u_{0,0}$ to both $u_{0,j}$ and $u_{2,j}$.

We now obtain a Hamilton from $u_{0,0}$ to any vertex of the form $u_{1,j}$, $j$
even, because both of the following permutations are automorphisms of $\mathrm{HTG}(3,n,\ell)$:
\begin{itemize}
\item $f(u_{i,j})=u_{i,j+2}$; and
\item $g(u_{i,j})=u_{1+1,j+1}$ for $i\in\{0,1\}$ and $g(u_{2,j})=u_{0,1+j+\ell}$.
\end{itemize}

\noindent Therefore, $\mathrm{HTG}(3,n,\ell)$ is Hamilton-laceable.     \end{proof}

\section{Cycle Structure}

We now look at cycles in honeycomb toroidal graphs with respect to two properties: girth and
cycle spectrum.  Throughout this section we use the important convention that the notation
$\mathrm{HTG}(m,n,\ell)$ always is in {\it normal form}, that is, $\ell\leq n/2$.  This
convention is possible because $\mathrm{HTG}(m,n,\ell)$ is isomorphic to $\mathrm{HTG}
(m,n,n-\ell)$.  Hence, the information given with respect to $\ell$ assumes $n\geq 2\ell$.

There are no odd length cycles because honeycomb toroidal graphs are bipartite.  All
$\mathrm{HTG}(m,n,\ell)$ contain 6-cycles ($K_4$ is not a honeycomb toroidal graph)
implying that the girth is either 4 or 6.  The next result handles the girth situation and is given
without its easy proof.

\begin{thm}\label{girth} The girth of $\mathrm{HTG}(m,n,\ell)$ is 6 with the following
exceptions for which the girth is 4:
\begin{itemize}
\item $n=4$;
\item $m=1, n>4$ and $\ell=3$;
\item $m=1,n>4,n\equiv 2(\mbox{\em mod }4)$ and $\ell=n/2$;
\item $m=1,n>4,n\equiv 0(\mbox{\em mod }4)$ and $\ell=\frac{n-2}{2}$; and
\item $m=2,n>4$ and $\ell\in\{0,2\}$. 
\end{itemize}
\end{thm}

We now consider the cycle spectrum property.
Recall that a graph is {\it even pancyclic} if it contains all possible even
length cycles from length 4 through $2\lfloor n/2\rfloor$, where $n$ is the order of the graph.
Given that connected bipartite Cayley graphs of valency at least 3 on abelian groups are even
pancyclic \cite{A3} and honeycomb toroidal graphs are Cayley graphs on groups that are close
to being abelian, we expect that the latter graphs should have a rich cycle spectrum.  

Cycles whose lengths are congruent to 2 modulo 4 are straightforward as the following
easily proved result indicates. 

\begin{thm}\label{four} The graph $\mathrm{HTG}(m,n,\ell)$ has cycles of length $L$ for all $L$
satisfying $L\equiv 2(\mbox{{\em mod }}4)$ and $6\leq L\leq mn$.
\end{thm}

From Theorems \ref{girth} and \ref{four}, we see that cycles whose lengths are multiples of 4
are of interest.  
Honeycomb toroidal graphs $\mathrm{HTG}(m,n,\ell)$ for $m\geq 3$ and $n>4$ have a simple 12-cycle
lying in three columns so that lengths 4 and 8 become the only possible missing values once it
is seen how to increase cycle lengths by 4 at a time.  The cycle spectrum problem was settled
for $\mathrm{HTG}(m,n,\ell)$, when $m\geq 3$, in \cite{S2}.  We summarize their results in
Table 1.  Any of the graphs not listed in the table are even pancyclic.  The missing even cycle
lengths from 4 through $mn-2$ are displayed in the right column.

This leaves the spectrum problem unsettled for $m=1$ and $m=2$.  The honeycomb toroidal
graphs for $m=1$ were seen to be crucial for the Hamilton laceability question so that they
are an interesting subclass.  As a side note, $\mathrm{HTG}(1,14,5)$ is the Heawood graph
so that the subclass contains well-known graphs.

\begin{center}
\begin{tabular}{|c|c|}
The graphs & Missing cycle lengths $L$\\ \hline
$\mathrm{HTG}(m,4,\ell),\mbox{even }m\geq 6$ & $L\equiv 0(\mbox{mod } 4)\mbox{ and }
4<L<2m$\\ \hline
$\mathrm{HTG}(m,4,\ell),\mbox{odd }m\geq 5$ & $L\equiv 0(\mbox{mod } 4)\mbox{ and }
4<L<2m+2$\\ \hline
$\mathrm{HTG}(m,n,\ell),m\geq 3,n=6,8$ & $L=4$\\ \hline
$\mathrm{HTG}(3,n,\ell),n\geq 10$ & $L=4$\\
$\ell\in\{\pm 1,\pm 3,\pm 5\}$ & \\ \hline
$\mathrm{HTG}(4,n,\ell),n\geq 10$ & $L=4$\\
$\ell\in\{0,\pm 2,\pm 4\}$ & \\ \hline
$\mathrm{HTG}(4,n,\ell),n\geq 10$ & $L=4,8$\\
$\ell\not\in\{0,\pm 2,\pm 4\}$ & \\ \hline
$\mathrm{HTG}(m,n,\ell),\mbox{even }m\geq6,n\geq 10$ & $L=4,8$\\ \hline
$\mathrm{HTG}(3,n,\ell),n\geq 10$ & $L=4,8$\\
$\ell\not\in\{\pm 1,\pm 3,\pm 5\}$ & \\ \hline
$\mathrm{HTG}(m,n,\ell),\mbox{odd }m\geq 5,n\geq 10$ & $L=4,8$\\ \hline
\end{tabular} 

\medskip

{\sc Table 1}
\end{center}

When $m=1$ it is easy to check that $\mathrm{HTG}(1,n,3)$ is even pancyclic.  For $\ell=5$,
$\mathrm{HTG}(1,n,5)$ is missing only a 4-cycle for even $n\geq 14$.   For $\ell=7$,
$\mathrm{HTG}(1,n,7)$ is missing only a 4-cycle for even $n\geq 18$.

For convenience we use only a single subscript describing the vertices when $m=1$.  The jump
edges then have the form $u_iu_{i+\ell}$ for all odd $i$, where the subscript arithmetic is
carried out modulo $n$.
For odd $\ell>7$, we have $n\geq 2\ell$ because $\mathrm{HTG}(1,n,\ell)$ is in normal
form.  If $n=2\ell$, then $\mathrm{HTG}(1,2\ell,\ell)$ is also a circulant graph and by the
main result of \cite{A3}, it is even pancyclic.  If $n=2\ell+2$, then $\mathrm{HTG}(1,2\ell+2,\ell)$
has girth 4 by Theorem \ref{girth}.  However, it may be missing an 8-cycle because $\ell>7$.
Finally, if $n>2\ell+2$, then $\mathrm{HTG}(1,n,\ell)$ contains the 12-cycle
\[u_0,u_1,u_2,u_{\ell+2},u_{\ell+3},u_{2\ell+3},u_{2\ell+2},u_{2\ell+1},u_{2\ell},u_{2\ell-1},
u_{\ell-1},u_{\ell},u_0.\]  Using methods from \cite{S2}, it is not hard to show that there are
cycles of all lengths that are multiples of 4 lying between 12 and $n$ inclusive.  This means the
only possible missing cycle lengths are 4 and 8.

Theorem \ref{girth} provides the information on 4-cycles and it is easy to see they all are
even pancyclic.  So 8-cycles turn out to be of the most interest.  The calculations used to
determine the $\mathrm{HTG}(1,n,\ell)$ which contain an 8-cycle are tedious and we give
two examples to show how it is done.  It is based on looking at the subpaths of the cycle
$C=u_0,u_1,u_2,\ldots,u_{n-1},u_0$.  We also may assume that the edge $u_1,u_{1+\ell}$
belongs to any 8-cycle because of rotational automorphisms of the graph.

If an 8-cycle has just a single subpath of $C$, then $\ell=7$ must hold and we are working
with $\ell>7$ so that there must be at least two subpaths.  If there are two subpaths, one
possibility is that one path has length 1 and the other has length 5.  There are several
subcases one of which is $$u_0,u_1,u_{1+\ell},u_{\ell},u_{2\ell},u_{2\ell-1},u_{2\ell-2},
u_{2\ell-3},u_0.$$  This implies that $\ell=(n+3)/3$.

The information for $m=1$ is contained in Table 2.  We move to the case of $m=2$.
From Theorem \ref{girth}, $\mathrm{HTG}(2,n,\ell)$ has girth 4 exactly when $\ell=0$
or 2, and in these cases they are easily seen to be even pancyclic.  For other values of
$\ell$, $$u_{0,1},u_{0,2},u_{0,3},u_{1,3},u_{1,4},u_{0,4+\ell},u_{0,3+\ell},u_{0,2+\ell},
u_{0,1+\ell},u_{0,\ell},u_{1,0},u_{1,1},u_{0,1}$$ is a 12-cycle.  It is straightforward to
verify that all other cycle lengths that are multiples of 4 and between 16 and $2n$, inclusive,
are realized.  So 8-cycles again become the key to fully determining the cycle spectrum.

It is trivial to find an 8-cycle when $n=6$ or $n=8$, so that we need consider only values
of $n\geq 10$.  An 8-cycle must intersect each  of the two column cycles in the same
number of subpaths.  By considering the possible intersections, 8-cycles occur as shown in
Table 2.  This table differs from Table 1 in that any honeycomb toroidal graph not mentioned
in the table for $m=1$ or $m=2$ is missing both 4-cycles and 8-cycles and no others of even
length in the feasible range.

\section{Paths And Diameter} 

The {\it diameter} of a connected graph is the maximum distance between pairs of distinct vertices
in the graph.  This parameter is of interest to anyone concerned with the propagation
of information throughout a network.  As this involves distances between vertices, we are
interested in shortest paths in honeycomb toroidal graphs.   The next two lemmas provide
useful information about shortest paths in honeycomb toroidal graphs.  Some terminology
is necessary before stating them.

\bigskip

\begin{center}
\begin{tabular}{|c|c|}
The graphs & Missing cycle lengths $L$\\ \hline
$\mathrm{HTG}(1,n,3), n\geq 6$ & none\\ \hline
$\mathrm{HTG}(1,n,n/2), n\equiv 2(\mbox{mod }4)$ & none\\ \hline
$\mathrm{HTG}(1,n,(n-2)/2), n\equiv 0(\mbox{mod }4)$ & none\\ \hline
$\mathrm{HTG}(2,n,\ell), \ell\in\{0,2\}$ & none\\ \hline
$\mathrm{HTG}(1,n,5), n\geq 14$ & $L=4$\\ \hline
$\mathrm{HTG}(1,n,7), n>14$ & $L=4$\\ \hline
$\mathrm{HTG}(1,n,\ell), n\equiv 2(\mbox{mod }4), n>14$ & $L=4$\\
odd $\ell\in\{(n-4)/2,(n-2)/4,(n+2)/4\}$ & \\ \hline
$\mathrm{HTG}(1,n,\ell), n\equiv 0(\mbox{mod }4), n>16$ & $L=4$\\ 
 odd $\ell\in\{(n-6)/2,(n-4)/4,n/4,(n+4)/4\}$ & \\ \hline
$\mathrm{HTG}(1,n,(n\pm 3)/3), n\equiv 0(\mbox{mod }6), n>18$ & $L=4$\\ \hline
$\mathrm{HTG}(1,n,\ell), n\equiv 4(\mbox{mod }6), n>10$ & $L=4$\\ 
$\ell\in\{(n-1)/3,(n+5)/3\}$ & \\ \hline
$\mathrm{HTG}(1,n,\ell), n\equiv 2(\mbox{mod }6), n>20$ & $L=4$\\
$\ell\in\{(n-5)/3,(n+1)/3\}$ & \\ \hline
$\mathrm{HTG}(2,n,4), n\geq 8$ & $L=4$\\ \hline
$\mathrm{HTG}(2,n,(n-4)/2), n\equiv 0(\mbox{mod }4), n>8$ & $L=4$\\ \hline
$\mathrm{HTG}(2,n,(n-2)/2), n\equiv 2(\mbox{mod }4), n>6$ & $L=4$\\ \hline

\end{tabular}

\bigskip

{\sc Table 2}
\end{center}

When talking about directions in which edges are traversed, travelling along a flat edge from
column $i$ to column $i+1$ is one direction and travelling from column $i+1$ to column $i$
is the other direction.  Similarly, the two directions for jump edges are from column 0 to
column $m-1$ and vice versa.

\begin{lemma}\label{LA} Every jump edge in a shortest path in $\mathrm{HTG}(m,n,\ell)$
is traversed in the same direction.
\end{lemma}
\begin{proof}  If a shortest path contains no jump edge, there is nothing to prove so let
$P$ be a shortest path in $\mathrm{HTG}(m,n,\ell)$ containing a jump edge.  Suppose
the first jump edge encountered when traversing $P$ is $u_{0,j}u_{m-1,j-\ell}$, that is,
we traverse it from column 0 to column $m-1$.  Suppose the next jump edge encountered
along $P$ has the form $u_{m-1,k}u_{0,k+\ell}$, that is, it is traversed from column $m-1$
to column 0. 

This implies that the subpath $P'$ of $P$ from $u_{m-1,j-\ell}$ to $u_{m-1,k}$ has no jump
edges and the second subscript has
changed from $j-\ell$ to $k$.  This is done only by vertical edges in various columns.
The change from $j$ to $k+\ell$ is the same as the change from $j-\ell$ to $k$.  Hence, we
may delete the subpath from $u_{0,j}$ to $u_{0,k+\ell}$ and replace it with the vertical
changes in $P'$ translated by $\ell$ projected onto column 0 .  This gives us a shorter walk
(some edges may be duplicated via the projection) from $u_{0,0}$ to the terminal vertex of $P$.
This is a contradiction to $P$ being a shortest path.

A similar contradiction arises if the traversals of two consecutive jump edges are reversed.
Therefore, all the jump edges in a shortest path go from column 0 to column $m-1$ or
vice versa.    \end{proof}

\begin{lemma}\label{LB} Let a shortest path $P$ in $\mathrm{HTG}(m,n,\ell)$ have jump edges.
if there are flat edges between the same two columns in $P$, they must be separated by a jump
edge.  In particular, if $P$ has no jump edges, then there is at most one flat edge between
two columns in $P$.
\end{lemma}
\begin{proof}  Let $u_{i,j}u_{i+1,j}$ and $u_{i,k}u_{i+1,k}$ be succesive appearances of
flat edges between columns $i$ and $i+1$, $0\leq i<m-1$.  Suppose that the first edge is
traversed from $u_{i,j}$ to $u_{i+1,j}$.  If there is no jump edge between $u_{i,j}u_{i+1,j}$
and the edge $u_{i,k}u_{i+1,k}$, then there are vertical edges taking the second subscript
from $j$ to $k$ no matter which direction $u_{i,k}u_{i+1,k}$ is traversed.  

In either case, remove the
subpath of $P$ from $u_{i,j}$ to $u_{i,k}$ and replace it with the projection of the vertical
edges onto column $i$.  This yields a shorter walk with the same terminal vertices which
is a contradiction.   Similar arguments work if the edge $u_{i,j}u_{i+1,j}$ is traversed in
the opposite direction.  The conclusion follows from this.   \end{proof}

\medskip

Consider the special graph $\mathrm{HTG}(m,n,0)$.  If we are looking for a shortest path
from $u_{0,0}$ to $u_{i,j}$, it is clear that we need vertical edges taking us to row $j$
and flat edges (the jump edge is also flat in this case) taking us to column $i$.  So if $i\leq m/2$,
we use flat edges in the direction left to right, and if $i>m/2$, we take a jump edge from
column 0 to column $m-1$ followed by flat edges from right to left.  We use vertical edges
as required to reach row $j$.  It is straightforward to obtain the diameter as shown in Table 3.

The preceding worked easily because the jump edges change the second subscript by zero.
Other values for the jump edges allow for big changes in shortest paths because a large
jump edge value allows large changes in the second subscript.  For example, suppose we
are trying to increase the second subscript as much as possible.  We can start a path at
$u_{0,0}$ and reach the vertex $u_{m-1,m-1}$ when we first reach column $m-1$.
We follow this with the edge $u_{m-1,m-1}u_{m-1,m}$ and then the jump edge
$u_{m-1,m}u_{0,m+\ell}$.  

We now have a path from $u_{0,0}$ to $u_{0,m+\ell}$ of length $2m$.  If instead we took
the path from $u_{0,0}$ to $u_{0,m+\ell}$ up column 0, it has length $m+\ell$.  Thus, if
$\ell>m$, we have a shorter path by using a jump edge.  Lemma \ref{LA} provides some
help because it tells us that if we use more than one jump edge in a shortest path, we must
use them in the same direction which forces many edges to be used between their appearances.

\smallskip

{\bf Research Problem 2}.  Determine the shortest paths between vertices in an arbitrary
$\mathrm{HTG}(m,n,\ell)$.

\smallskip

The diameters of a few honeycomb toroidal graphs
have been determined in \cite{S1,Y1,Y3} and we summarize their results in the following
table.  Note that \cite{Y1} corrects an error for the diameter of $\mathrm{HTG}(m,2m,m)$
given in \cite{S1}.
\bigskip

\begin{center}
\begin{tabular}{|c|c|}
The graphs & diameter\\ \hline
$\mathrm{HTG}(m,6m,3m)$ & $2m$\\ \hline
$\mathrm{HTG}(m,2m,m), m\geq 2, m\equiv 1,4(\mbox{mod }6)$ & $\lfloor 4m/3\rfloor$\\ \hline
$\mathrm{HTG}(m,2m,m), m\geq 2, m\equiv 0,2,3,5(\mbox{mod }6)$ & $\lceil 4m/3\rceil$\\ \hline
$\mathrm{HTG}(m,n,0), m\mbox{ even }, m\geq n-2$ & $m$\\ \hline
$\mathrm{HTG}(m,n,0), m\mbox{ even }, m<n-2$ & $(n+m)/2$\\ \hline
$\mathrm{HTG}(m,n,\ell), m\geq n/2,\ell\equiv n-m(\mbox{mod }n)$ & max$\{m,\lfloor(2m+n+1)/3\rfloor$\\ \hline
\end{tabular}

\bigskip

{\sc Table 3}
\end{center}

{\bf Research Problem 3}.  Determine the diameter of
$\mathrm{HTG}(m,n,\ell)$ in terms of the parameters $m,n\mbox{ and }\ell$.

\smallskip

The preceding problem undoubtedly has many subcases as the value of the jump varies.
Lemmas \ref{LA} and \ref{LB} allow us to determine that the diameter of
$\mathrm{HTG}(1,n,\ell)$ is $2\lfloor n/\ell\rfloor+1$ whenever
$\ell\leq\sqrt{n}$.  We shall not present the tedious proof of this fact, but mention it
just to indicate the kinds of complications that likely arise in considering the preceding
problem. 

\section{Automorphisms}

Honeycomb toroidal graphs are Cayley graphs \cite{A4} on a generalized dihedral group.
This means they are vertex-transitive.  As mentioned earlier, $\mathrm{HTG}(1,14,5)$ is the
Heawood graph and its automorphism group has order 336 in spite of the graph having only
14 vertices.  On the other hand, the automorphism group of $\mathrm{HTG}(1,14,3)$ has
order only 28.  So we see there may be wide variations in the automorphism groups of these
graphs.  This suggests the next problem.

\smallskip

{\bf Research Problem 4}.  Determine the automorphism group of an arbitrary
$\mathrm{HTG}(m,n,\ell)$ in terms of the parameters $m,n\mbox{ and }\ell$.

\smallskip
 
Given a family of Cayley graphs, there is interest in determining those with minimal
automorphism groups.  In this case that means those that are GRRs, that is, those
for which $|\mathrm{Aut}(\mathrm{HTG}(m,n,\ell)|=mn$. 

\smallskip

{\bf Research Problem 5}.  Determine when $\mathrm{HTG}(m,n,\ell)$ is a GRR, that is,
$|\mathrm{Aut}(\mathrm{HTG}(m,n,\ell))|=mn$.

\smallskip

Little is known about the preceding question.  One result in this direction comes from
\cite{C3} in which the following result is proved.

\begin{thm} The graph $\mathrm{HTG}(1,n,\ell)$ in normal form is a $\mathrm{GRR}$ if and only if
$n\geq 18, \ell<n/2$ and the following all hold:
\begin{itemize}
\item $(\ell+1)^2/4\not\equiv 1(\mbox{\em{mod} }n/2)$;
\item $(\ell-1)^2/4\not\equiv 1(\mbox{\em{mod} }n/2)$; and
\item $(\ell^2-1)/4\not\equiv -1(\mbox{\em{mod} }n/2)$.
\end{itemize}
\end{thm}

\section{Conclusion}

The family of graphs under discussion are of interest for several reasons and we have looked at
them primarily from a graph theoretic viewpoint.  There has been considerable work done on
algorithmic aspects of honeycomb toroidal graphs.  Most of the concern is with routing,
broadcasting, bisection width, semigroup computation and cost.  Again, most of the work has
dealt with the special honeycomb toroidal graphs introduced in \cite{S1} and their extensions.
So there is room for research for the entire family of honeycomb toroidal graphs.  There is
background for the algorithmic work in \cite{C5,C4,L2,M2,S3,S1}.

We also have presented some specific research problems that we find interesting.  This is a
family of graphs worthy of much further investigation.

\end{document}